\documentclass[a4paper,twocolumn,aps,10pt,pra,twoside,hidelinks]{revtex4-1}
  \usepackage{pstricks-add,amsthm,framed,subfigure,hyperref,breakurl}
  \newtheorem*{thm}{Theorem}
  \newtheorem{lem}{Lemma}
  \usepackage[left=2.0cm, right=2.0cm, top=2.0cm, bottom=3.0cm,a4paper]{geometry}
  \psset{unit=4mm,linewidth=0.3pt}
\begin{document}

  \title{Smallest Squared Squares}
  \author{Lorenz Milla}
  \email{lorenz.milla@gmx.net}
  \affiliation{Heidelberg, Germany}
  \date{July 9, 2013}

\begin{abstract} 
In this paper we have a look at squared squares with small integer sidelengths, where the only restriction is that any two subsquares of the same size are not allowed to share a full border. We prove that there are exactly two such squared squares (and their mirrored versions) up to and including size 17x17. They are shown in Figure 1 (page \pageref{fig1116}).
\end{abstract}

\maketitle

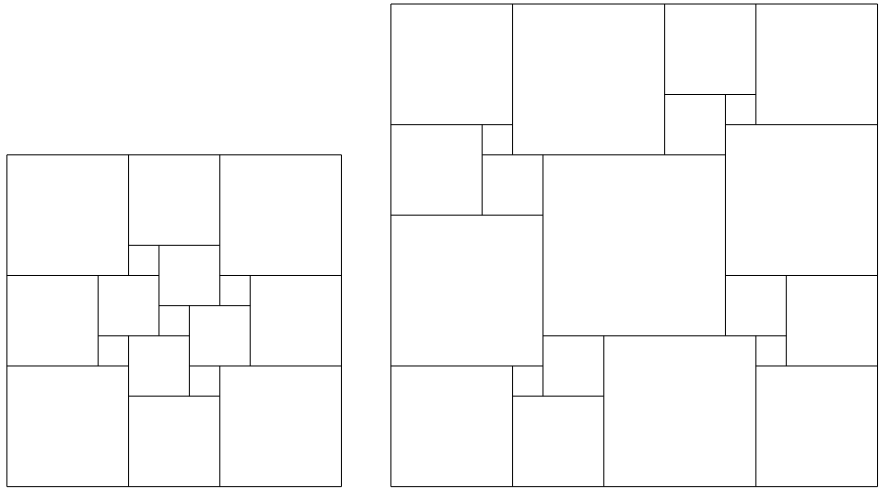
\begin{figure*}[!ht]
\centering
\subfigure{\begin{pspicture}(0,0)(11,11)
        \psline(0,0)(0,11)
        \psline(0,0)(11,0)
        \psline(4,0)(4,5)
        \psline(4,4)(0,4)
        \psline(4,3)(7,3)
        \psline(7,0)(7,4)
        \psline(6,4)(11,4)
        \psline(6,3)(6,6)
        \psline(6,5)(3,5)
        \psline(3,4)(3,7)
        \psline(0,7)(5,7)
        \psline(5,5)(5,8)
        \psline(4,8)(7,8)
        \psline(4,7)(4,11)
        \psline(7,6)(7,11)
        \psline(5,6)(8,6)
        \psline(8,4)(8,7)
        \psline(7,7)(11,7)
        \psline(11,0)(11,11)
        \psline(11,11)(0,11)
    \end{pspicture}}\qquad
\subfigure{\begin{pspicture}(0,0)(16,16)
        \psline(0,0)(0,16)
        \psline(0,16)(16,16)
        \psline(16,16)(16,0)
        \psline(16,0)(0,0)
				\psline(4,0)(4,4)
        \psline(3,9)(3,12)
        \psline(4,11)(4,16)
        \psline(5,3)(5,11)
        \psline(7,0)(7,5)
        \psline(4,3)(7,3)
        \psline(9,11)(9,16)
        \psline(11,5)(11,13)
        \psline(12,0)(12,5)
        \psline(12,12)(12,16)
        \psline(13,4)(13,7)
        \psline(0,4)(5,4)
        \psline(0,9)(5,9)
        \psline(0,12)(4,12)
        \psline(3,11)(11,11)
        \psline(5,5)(13,5)
        \psline(12,4)(16,4)
        \psline(11,7)(16,7)
        \psline(11,12)(16,12)
        \psline(9,13)(12,13)
    \end{pspicture}}
\label{fig1116}
\caption{Smallest Squared Squares (11x11 and 16x16)}
\end{figure*}

\section{Introduction}
\noindent\textbf{Definitions.} The general definitions, formulated by S. Anderson (priv. comm. and \cite{anderson2013}): ``A \emph{squared rectangle} is a rectangle dissected into a finite number, two or more, of squares, called the \emph{elements} of the dissection. If no two of these squares have the same size the squared rectangle is called \emph{perfect}, otherwise it is \emph{imperfect}. The \emph{order} of a squared rectangle is the number of constituent squares. The case in which the squared rectangle is itself a square is called a \emph{squared square}. The dissection is \emph{simple} if it contains no smaller squared rectangle, otherwise it is \emph{compound}.''

Additionally, we call a dissection \emph{trivial} if it contains two elements of the same size which share a full border. In this paper we analyze \emph{nontrivial} squared squares, and for brevity we say ``squared square'' instead of ``nontrivial squared square''. Note that all simple perfect squared squares, all compound perfect squared squares and all simple imperfect squared squares are nontrivial, but it is possible to find many trivial compound imperfect squared squares with small integer size, for example a square of size 2x2, dissected into four squares of size 1x1.

By a result of M. Dehn \cite{dehn1903}, a rectangle can be tiled by a finite number of squares if and only if the rectangle has commensurable sides. From commensurability it follows that the squared rectangle's sides and elements can all be given in integers without any common divisor. This explains what we call a \emph{small} squared square.

This paper is structured as follows: in Sec.~\ref{Sect2}, we prove some basic lemmata and show that there is exactly one squared square up to and including size 11x11. In Sec.~\ref{Sect3} we prove there is no squared square of size 12x12 to 15x15. In Sec.~\ref{Sect4} we prove that there is exactly one squared square of size 16x16. In Sec.~\ref{Sect5} we prove there is no squared square of size 17x17. In Sec.~\ref{Sect6} we prove the final theorem as formulated in the abstract.

\textbf{Remark.} The techniques used in this paper have been used by Ian Gambini in his doctoral thesis \cite{gambini1999}. Gambini analyzed \emph{perfect} squa\-red squares and enumerated all perfect squared squares up to size 130x130 with help of the computer, and showed that there are (up to and including size 130x130) exactly five perfect squared squares: three of size 110, one of size 112, and one of size 120.

\section{Up to size 11x11}
\label{Sect2}

\begin{lem}\label{lem1}All squares at the border of a dissection have at least size three.\end{lem}
\begin{proof}We prove that sizes 1 and 2 are impossible:

A square with size 1 in a \emph{corner} leads to $a\geq 2$ and thus to $b=1$ which is not allowed:
    
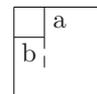
\begin{center}
    \begin{pspicture}(0,0)(3,3)
      \rput(0.5,1.5){b}
      \rput(1.5,2.5){a}
      \psline(0,0)(0,3)
      \psline(0,3)(3,3)
      \psline(1,2)(1,3)
      \psline(0,2)(1,2)
      \psline[linestyle=dashed](1,2)(1,1)
    \end{pspicture}
    \end{center}

If a square with size 1 is at the \emph{border} (and not in the corner), it has two neighbors $a$ and $b$ at the border:
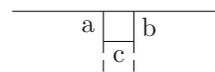
\begin{center}
    \begin{pspicture}(0,1)(7,3)
      \rput(2.5,2.5){a}
      \rput(4.5,2.5){b}
      \rput(3.5,1.5){c}      
        \psline(0,3)(7,3)
        \psline(3,2)(3,3)
        \psline(3,2)(4,2)
        \psline(4,2)(4,3)
        \psline[linestyle=dashed](4,1)(4,2)
        \psline[linestyle=dashed](3,1)(3,2)
    \end{pspicture}
\end{center}
Because $a\geq 2$ and $b\geq 2$ it follows $c = 1$, which is not allowed.

A square with size 2 in a \emph{corner} leads to $a\geq 3$ and thus to $b=c=1$ which is not allowed:
    
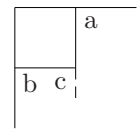
\begin{center}
    \begin{pspicture}(0,0)(4,4)
      \rput(0.5,1.5){b}
      \rput(1.5,1.5){c}
      \rput(2.5,3.5){a}
      \psline(0,0)(0,4)
      \psline(0,4)(4,4)
      \psline(2,2)(2,4)
      \psline(0,2)(2,2)
      \psline[linestyle=dashed](2,2)(2,1)
    \end{pspicture}
    \end{center}

If a square with size 2 is at the \emph{border} (and not in the corner), it also has two neighbors $a$ and $b$ at the border:	
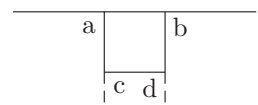
\begin{center}
    \begin{pspicture}(0,0)(8,3)
      \rput(2.5,2.5){a}
      \rput(5.5,2.5){b}
      \rput(3.5,0.5){c}
      \rput(4.5,0.5){d}
        \psline(0,3)(8,3)
        \psline(3,1)(3,3)
        \psline(3,1)(5,1)
        \psline(5,1)(5,3)
        \psline[linestyle=dashed](5,0)(5,1)
        \psline[linestyle=dashed](3,0)(3,1)
    \end{pspicture}
\end{center}
It is already shown (see above) that $a\neq 1$ and $b\neq 1$. Because they can't have the same size as their middle neighbor, it follows $a\geq 3$ and $b\geq 3$. It follows $c\leq 2$ and because $c\neq 2$ it follows $c = 1$ and thus $d=1$, wich is not allowed.
\end{proof}

\begin{lem}\label{lem2}Squares in a corner of a dissection have at least size four.\end{lem}
\begin{proof} Combined with Lemma \ref{lem1} we just have to deal with one case: Size 3 in a corner.
    
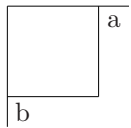
\begin{center}
    \begin{pspicture}(0,0)(4,4)
      \rput(0.5,0.5){b}
      \rput(3.5,3.5){a}
      \psline(0,0)(0,4)
      \psline(0,4)(4,4)
      \psline(3,1)(3,4)
      \psline(0,1)(3,1)
    \end{pspicture}
    \end{center}
But here, both $a$ and $b$ are at the border of the dissection, so (from Lemma \ref{lem1}) we know that $a\geq 3$ and $b\geq 3$. But we already have a 3x3 in the corner, thus $a\geq 4$ and $b\geq 4$. This leads to an overlap, thus we can't have a square of size 3x3 in a corner.
\end{proof}

\begin{lem}\label{lem3}In every squared square, at least two squares touching the borders don't sit in the corners. Equivalent: At least six squares are touching the borders.\end{lem}
\begin{proof}
Let's start with a border touched by only two squares (which must then have different sizes $a<b$):
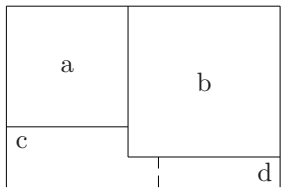
    \begin{center}
    \begin{pspicture}(0,0)(9,6)
      \rput(6.5,3.5){b}
      \rput(2,4){a}
      \rput(0.5,1.5){c}
      \rput(8.5,0.5){d}
      \psline(0,0)(0,6)
      \psline(0,6)(9,6)
      \psline(9,6)(9,0)
      \psline(9,1)(4,1)
      \psline(4,1)(4,6)
      \psline(0,2)(4,2)
      \psline[linestyle=dashed](5,0)(5,1)
    \end{pspicture}
    \end{center}
If the square in position $c$ would touch the bottom border, it would be of size $b$, this would lead to an overlap. Thus in position $c$ there has to be a border square which doesn't touch the bottom corner.
Next we can have two cases: If in position $d$ there is also a border square which doesn't touch the bottom border, we're finished. If not, we must have $d=a$ because the whole rectangle is a square (and thus $a+b=b+d$). But then the bottom border can't consist of only two squares (because we have $d=a$ in the right bottom corner, and in the left bottom corner we don't have size $b$), which means we proved the Lemma.
\end{proof}
\begin{lem}\label{least11}Every squared square has at least size 11.\end{lem}
\begin{proof}
From Lemma \ref{lem3} we deduce that on at least one side we have three squares -- two in the corners, one in the middle. Their minimal size is given by Lemma \ref{lem1} and \ref{lem2}, so the minimum size of a squared square is 4+3+4=11.
\end{proof}
\begin{lem}\label{lem11}Besides the squared square shown in Figure 1 (and the mirrored version of it), there is no other squared square of size 11.\end{lem}
\begin{proof}
Because of Lemma \ref{lem3} there must be one border with at least three elements. The only possible way to dissect a sidelength of 11 into three or more squares is 4+3+4. Thus every squared square of size 11 must contain (at least) one border with 4+3+4:
    
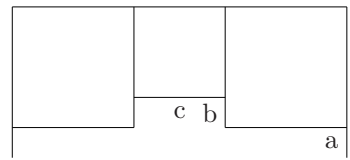
\begin{center}
    \begin{pspicture}(0,0)(11,5)
      \rput(10.5,0.5){a}
      \rput(6.5,1.5){b}
      \rput(5.5,1.5){c}
      \psline(0,0)(0,5)
      \psline(0,5)(11,5)
      \psline(11,5)(11,0)
      \psline(0,1)(4,1)
      \psline(4,1)(4,5)
      \psline(4,2)(7,2)
      \psline(7,5)(7,1)
      \psline(7,1)(11,1)
    \end{pspicture}
    \end{center}
If $a$ touched the bottom border, then it would be $a=11-4=7$, but then $b=c=1$, which is not allowed. Thus follows that $a$ doesn't touch the bottom border, and so we must have $a=3$ and a 4x4 square underneath it. Analogously, we deduce that the whole border has to be 4+3+4 on every side:
    
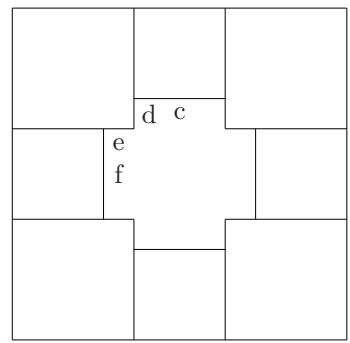
\begin{center}
    \begin{pspicture}(0,0)(11,11)
      \rput(4.5,7.5){d}
      \rput(3.5,6.5){e}
      \rput(3.5,5.5){f}
      \rput(5.5,7.5){c}
      \psline(0,0)(0,11)
      \psline(0,11)(11,11)
      \psline(11,11)(11,0)
      \psline(11,0)(0,0)
      \psline(0,4)(4,4)
      \psline(4,4)(4,0)
      \psline(7,0)(7,4)
      \psline(7,4)(11,4)
      \psline(7,7)(11,7)
      \psline(7,7)(7,11)
      \psline(0,7)(4,7)
      \psline(4,7)(4,11)
      \psline(4,8)(7,8)
      \psline(8,4)(8,7)
      \psline(4,3)(7,3)
      \psline(3,4)(3,7)
    \end{pspicture}
    \end{center}
We now have $d\leq 3$. But if $d=3$ we would have $e=f=1$, which is not allowed. So two options remain: $d=1$ or $d=2$.
\begin{itemize}
	\item If $d=1$, then $c=2$, then $e=2$.
	\item If $d=2$, then $e=1$.
\end{itemize}
Because we don't distinguish mirrored versions we can assume that $d=1$ and $e=2$.
The remaining center part:
    
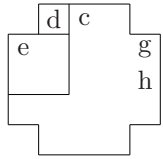
\begin{center}
    \begin{pspicture}(0,0)(5,5)
      \rput(1.5,4.5){d}
      \rput(0.5,3.5){e}
      \rput(2.5,4.5){c}
      \rput(4.5,3.5){g}
      \rput(4.5,2.5){h}
      \psline(0,1)(0,4)
      \psline(0,4)(2,4)
      \psline(1,4)(1,5)
      \psline(1,5)(4,5)
      \psline(4,5)(4,4)
      \psline(4,4)(5,4)
      \psline(5,4)(5,1)
      \psline(5,1)(4,1)
      \psline(4,1)(4,0)
      \psline(4,0)(1,0)
      \psline(1,0)(1,1)
      \psline(1,1)(0,1)
      \psline(0,2)(2,2)
      \psline(2,2)(2,5)
    \end{pspicture}
    \end{center}
We deduce: $c=2$, $g=1$, $h=2$, etc. and get the smallest squared square from Figure 1 (page \pageref{fig1116}). Because this construction was the only way (apart from the symmetry in $d$ and $e$), this is the only solution in size 11x11.
\end{proof}

\section{Size 12x12 to 15x15}
\label{Sect3}
\begin{lem}\label{no12}There is no squared square with size 12x12.\end{lem}
\begin{proof}
After Lemma 3, there must be one border with at least 3 elements. These have to be at least 4+3+4, where one of them is increased by 1. The center one can't be increased (4+4+4 is not allowed), so we have 4+3+5 (or, symmetrically, 5+3+4) as the only starting possibility:
    
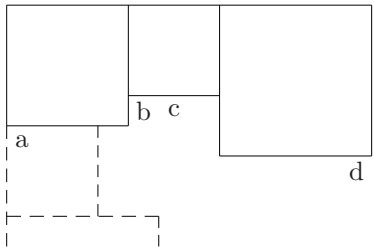
\begin{center}
    \begin{pspicture}(0,-1)(12,7)
      \rput(0.5,2.5){a}
      \rput(4.5,3.5){b}
      \rput(5.5,3.5){c}
      \rput(11.5,1.5){d}
      \psline(0,3)(0,7)
      \psline(0,7)(12,7)
      \psline(12,7)(12,2)
      \psline(12,2)(7,2)
      \psline(7,2)(7,7)
      \psline(7,4)(4,4)
      \psline(4,3)(4,7)
      \psline(0,3)(4,3)
      \psline[linestyle=dashed](0,-1)(0,3)
      \psline[linestyle=dashed](0,0)(5,0)
      \psline[linestyle=dashed](3,0)(3,3)
      \psline[linestyle=dashed](5,-1)(5,0)
    \end{pspicture}
    \end{center}
Two cases: Either $a$ reaches the bottom with $a=12-4=8$, or it doesn't (then $a=3$, with a 5x5 below). But $a=8$ is impossible, it would overlap the existing 5x5 square. Thus it has to be $a=3$ with a 5x5 below (see dashed lines). Then we have two cases for $d$: either $d$ reaches the bottom with $d=12-5=7$, or it doesn't (then $d=3$ with a 4x4 below). But if $d=7$ then $c=2$ and $b=1$ and the rest is a rectangle with size $2x3$ and can't be filled properly. Thus it has to be $d=3$ with a 4x4 below:
    
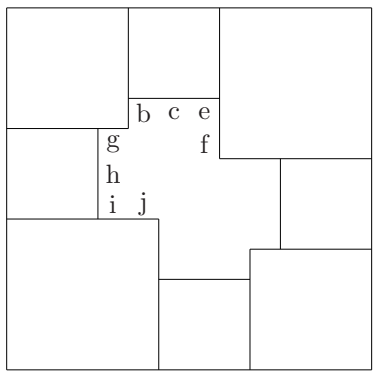
\begin{center}
    \begin{pspicture}(0,-5)(12,7)
      \rput(4.5,3.5){b}
      \rput(5.5,3.5){c}
      \rput(6.5,3.5){e}
      \rput(6.5,2.5){f}
      \rput(3.5,2.5){g}
      \rput(3.5,1.5){h}
      \rput(3.5,0.5){i}
      \rput(4.5,0.5){j}
      \psline(0,-5)(0,7)
      \psline(0,7)(12,7)
      \psline(12,7)(12,-5)
      \psline(12,-5)(0,-5)
      \psline(12,2)(7,2)
      \psline(7,2)(7,7)
      \psline(7,4)(4,4)
      \psline(4,3)(4,7)
      \psline(0,3)(4,3)
      \psline(0,0)(5,0)
      \psline(3,0)(3,3)
      \psline(5,-5)(5,0)
      \psline(5,-2)(8,-2)
      \psline(8,-5)(8,-1)
      \psline(8,-1)(12,-1)
      \psline(9,-1)(9,2)
    \end{pspicture}
    \end{center}
The options for $b\leq 3$ are: If $b=2$ then $e=f=1$ which is not allowed. If $b=3$ then $g=h=1$ which is not allowed. Thus it has to be $b=1$ and thus $cef=2$ and thus $gh=2$. This leads to $i=j=1$, which is not allowed.
\end{proof}

\begin{lem}\label{lemx3y}It is not possible to decompose a border with $x+3+y$ where $x\geq 5$ and $y\geq 5$.\end{lem}
\begin{proof}
We would have this situation:
    
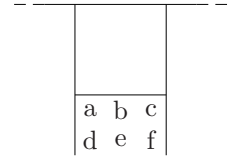
\begin{center}
    \begin{pspicture}(0,0)(7,6)
      \rput(2.5,2.5){a}
      \rput(3.5,2.5){b}
      \rput(4.5,2.5){c}
      \rput(2.5,1.5){d}
      \rput(3.5,1.5){e}
      \rput(4.5,1.5){f}
      \psline(2,1)(2,6)
      \psline(5,1)(5,6)
      \psline(2,3)(5,3)
      \psline(1,6)(6,6)
      \psline[linestyle=dashed](0,6)(1,6)
      \psline[linestyle=dashed](6,6)(7,6)
    \end{pspicture}
    \end{center}
$a=3$ is not allowed (top neighbor). If $a=2$ then $c=f=1$ which is ot allowed. If $a=1$ then $b=2$ and thus $d=1$ which is not allowed. So this gap can't be filled.
\end{proof}
\begin{lem}\label{no13}There is no squared square with size 13x13.\end{lem}
\begin{proof}
We start again with one of the borders consisting of at least three elements (such a border exists according to Lemma 3).
The smallest sidelength for which a border with four elements is possible is 4+3+4+5=16.
Thus in every squared square with size 13x13 there exists at least one border with exactly three elements.
For this border we have several possibilities:
If one of the corners in this border has size 4, then we have 9 left for the other two. This leads to 4+3+6 or 4+4+5 or 4+5+4, but 4+4+5 is not allowed.
If both corners are larger than 4, we have just one possibility, 5+3+5. But because of Lemma \ref{lemx3y} this doesn't make a square.
So we have two options left:
\begin{enumerate}
	\item start with 4+5+4
	\item start with 4+3+6
\end{enumerate}
Try to start with 4+5+4 first:
    
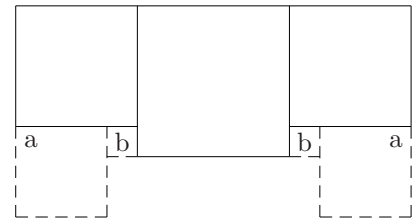
\begin{center}
    \begin{pspicture}(0,0)(13,7)
      \rput(0.5,2.5){a}
      \rput(3.5,2.5){b}
      \rput(9.5,2.5){b}
      \rput(12.5,2.5){a}
      \psline(0,3)(0,7)
      \psline(0,7)(13,7)
      \psline(13,7)(13,3)
      \psline(13,3)(9,3)
			\psline(9,2)(9,7)
			\psline(9,2)(4,2)
			\psline(4,2)(4,7)
			\psline(0,3)(4,3)
			\psline[linestyle=dashed](0,3)(0,0)
			\psline[linestyle=dashed](0,0)(3,0)
			\psline[linestyle=dashed](3,0)(3,3)
			\psline[linestyle=dashed](3,2)(4,2)
			\psline[linestyle=dashed](9,2)(10,2)
			\psline[linestyle=dashed](10,3)(10,0)
			\psline[linestyle=dashed](10,0)(13,0)
			\psline[linestyle=dashed](13,3)(13,0)
		\end{pspicture}
    \end{center}
In positions $a$ we must have $a=3$, because $a=4$ is not allowed and border elements must have at least size 3. Underneath the $a$ squares, there is a gap of 6 left, which can't be dissected into more than one square, so in both lower corners have to be 6x6 squares. But between these is a gap with size 1x6 which can't be filled. This means that the dissection 4+5+4 isn't part of any squared square.

Next we try to start with 4+3+6:
    
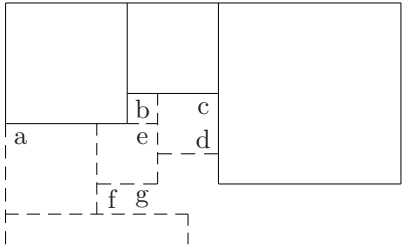
\begin{center}
    \begin{pspicture}(0,0)(13,8)
      \rput(0.5,3.5){a}
      \rput(4.5,4.5){b}
      \rput(6.5,4.5){c}
      \rput(6.5,3.5){d}
			\rput(4.5,3.5){e}
			\rput(3.5,1.5){f}
			\rput(4.5,1.5){g}
			\psline(0,4)(0,8)
      \psline(0,8)(13,8)
      \psline(13,8)(13,2)
      \psline(13,2)(7,2)
			\psline(7,2)(7,8)
			\psline(7,5)(4,5)
			\psline(4,8)(4,4)
			\psline(4,4)(0,4)
			\psline[linestyle=dashed](0,4)(0,0)
			\psline[linestyle=dashed](0,1)(6,1)
			\psline[linestyle=dashed](3,1)(3,4)
			\psline[linestyle=dashed](3,2)(5,2)
			\psline[linestyle=dashed](5,2)(5,5)
			\psline[linestyle=dashed](4,4)(5,4)
			\psline[linestyle=dashed](5,3)(7,3)
			\psline[linestyle=dashed](6,1)(6,0)
		\end{pspicture}
    \end{center}
Then on the left border we have $13-4=9$ left, this can't be done with one square (because it would overlap with the existing 6x6 square), so we have to do it with two squares. As already shown, 4+5+4 is not possible, so we must have $a=3$ with a 6x6 below (dashed lines).
Next we observe that $b\leq 3$, but $b=3$ is not allowed and if $b=2$ then $c=d=1$ which is not allowed. Thus we have $b=1$ and $cd=2$, thus $e=2$ and $f=g=1$, which is not allowed. This means that 4+3+6 doesn't lead to a squared square either.
\end{proof}

\begin{lem}\label{lem64}If a 4x4 square is in the corner, then no neighbor can have size 6x6 or bigger.\end{lem}
\begin{proof}
If we had such a situation, then we would have $a\geq 3$ because it is at the border:
    
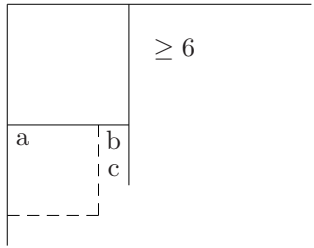
\begin{center}
    \begin{pspicture}(0,-1)(10,7)
      \rput(0.5,2.5){a}
			\rput(5.5,5.5){$\geq 6$}
      \rput(3.5,2.5){b}
      \rput(3.5,1.5){c}
      \psline(0,3)(0,7)
      \psline(0,7)(10,7)
			\psline(0,3)(4,3)
      \psline(4,7)(4,1)
      \psline(0,3)(0,-1)
			\psline[linestyle=dashed](0,0)(3,0)
			\psline[linestyle=dashed](3,0)(3,3)
		\end{pspicture}
    \end{center}
But $a=4$ is not allowed, thus follows $a=3$. But then $b=c=1$ which is not allowed.
\end{proof}

\begin{lem}\label{no14}There is no squared square with size 14x14.\end{lem}
\begin{proof}
We start again with one of the borders consisting of three elements (such a border exists according to Lemma 3).
First we list all decompositions of a 14-border with 3 or more elements:
\begin{itemize}
	\item corner: 4, rest of border: 10 $\Rightarrow$ 4+3+7 or 4+6+4
	\item corner: 5, rest of border: 9 $\Rightarrow$ 5+3+6 or 5+4+5
	\item corner: 6, rest of border: 8 $\Rightarrow$ 6+3+5
	\item corner: 7, rest of border: 7 $\Rightarrow$ 7+3+4
\end{itemize}
But (because of Lemma \ref{lem64}) 4+6+4 is impossible and (because of Lemma \ref{lemx3y}) 5+3+6 is impossible. So (because of symmetry) we have two options left: 4+3+7 and 5+4+5.

First we try to start with 4+3+7:
	    
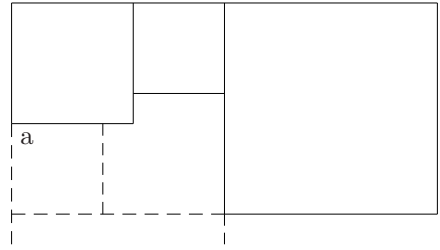
\begin{center}
			\begin{pspicture}(0,0)(14,8)
				\rput(0.5,3.5){a}
				\psline(0,4)(0,8)
				\psline(0,8)(14,8)
				\psline(14,8)(14,1)
				\psline(14,1)(7,1)
				\psline(7,1)(7,8)
				\psline(7,5)(4,5)
				\psline(4,4)(4,8)
				\psline(4,4)(0,4)
				\psline[linestyle=dashed](0,4)(0,0)
				\psline[linestyle=dashed](3,1)(3,4)
				\psline[linestyle=dashed](0,1)(7,1)
				\psline[linestyle=dashed](7,0)(7,1)
			\end{pspicture}
			\end{center}
     On the left border, the square starting in position $a$ can't reach the bottom ($a\neq 10$), because it would overlap the existing 7x7 square. As we have only two possible decompositions of the 14-border into at least three elements (see above) and because 5+4+5 doesn't fit here (we have 4x4 in the corner already), the only possible decompostion of the left border is 4+3+7 (dashed lines). But then we would get a 7x7 substructure in the top/left quarter of the dissection, and in Lemma \ref{least11} we showed that such doesn't exist. Thus 4+3+7 doesn't lead to a squared square.

This leaves the last decomposition, 5+4+5:
	    
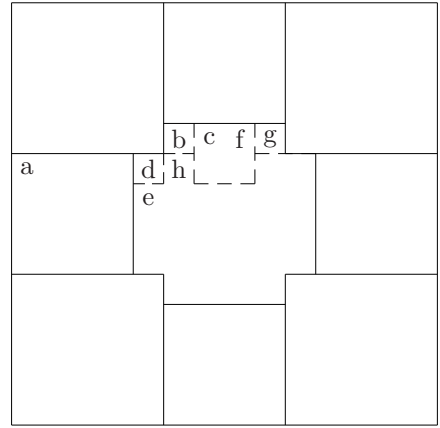
\begin{center}
			\begin{pspicture}(0,0)(14,14)
				\rput(0.5,8.5){a}
				\rput(5.5,9.5){b}
				\rput(6.5,9.5){c}
				\rput(4.5,8.5){d}
				\rput(4.5,7.5){e}
				\rput(7.5,9.5){f}
				\rput(8.5,9.5){g}
				\rput(5.5,8.5){h}				
				\psline(0,0)(0,14)
				\psline(0,14)(14,14)
				\psline(14,14)(14,0)
				\psline(14,0)(0,0)
				\psline(0,5)(5,5)
				\psline(5,5)(5,0)
				\psline(5,4)(9,4)
				\psline(9,0)(9,5)
				\psline(9,5)(14,5)
				\psline(10,5)(10,9)
				\psline(9,9)(14,9)
				\psline(9,9)(9,14)
				\psline(9,10)(5,10)
				\psline(5,9)(5,14)
				\psline(5,9)(0,9)
				\psline(4,5)(4,9)
				\psline[linestyle=dashed](4,8)(5,8)
				\psline[linestyle=dashed](5,8)(5,9)
				\psline[linestyle=dashed](5,9)(6,9)
				\psline[linestyle=dashed](6,10)(6,8)
				\psline[linestyle=dashed](8,8)(6,8)
				\psline[linestyle=dashed](8,8)(8,10)
				\psline[linestyle=dashed](8,9)(10,9)
			\end{pspicture}
			\end{center}
Here $a$ can't reach the bottom because then we'd have $b=c=1$. And because 5+4+5 is the only remaining decomposition with more than two elements, the square must be surrounded by 5+4+5 on each side (see above). Now we must have $b\leq 2$, otherwise $d=e=1$. But if $b=2$ then $f=1$ and $g$ can't be filled. So we must have $b=1$. A symmetrical argument leads to $d=1$ and $g=1$, so $cf=2$. But now $h$ can't be filled.
\end{proof}

\begin{lem}\label{lem43x}It is not possible to complete a border decomposed $4+3+x$ where $x\geq 6$ to a full squared square.\end{lem}
\begin{proof}
We would have a situation like this::
	    
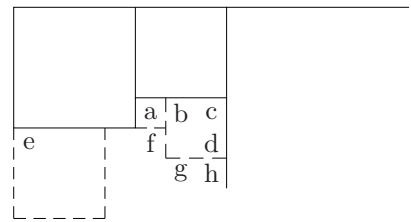
\begin{center}
			\begin{pspicture}(0,2)(13,9)
				\rput(4.5,5.5){a}
				\rput(5.5,5.5){b}
				\rput(6.5,5.5){c}
				\rput(6.5,4.5){d}
				\rput(0.5,4.5){e}
				\rput(4.5,4.5){f}
				\rput(5.5,3.5){g}
				\rput(6.5,3.5){h}
				\pspolygon(0,9)(7,9)(7,6)(4,6)(4,5)(0,5)(0,9)
				\psline(4,6)(4,9)
				\psline(7,6)(7,3)
				\psline(7,9)(13,9)
				\psline[linestyle=dashed](0,2)(0,5)
				\psline[linestyle=dashed](0,2)(3,2)
				\psline[linestyle=dashed](3,2)(3,5)
				\psline[linestyle=dashed](4,5)(5,5)
				\psline[linestyle=dashed](5,6)(5,4)
				\psline[linestyle=dashed](5,4)(7,4)
			\end{pspicture}
			\end{center}
First we observe that $e\neq 4$.
If we had $a=2$ then $c=d=1$. If we had $a=3$ then $e=3$ (dashed) and the gap between $a$ and $e$ couldn't be filled. Thus we have $a=1$ and $b=2$ (dashed lines).
Now we have two possibilities: Either $e=3$ (but then $f=2$, $g=1$, and $h$ can't be filled) or $e=5$ (but then also $g=1$ and $h$ can't be filled).\end{proof}

\begin{lem}\label{no15}There is no squared square with size 15x15.\end{lem}
\begin{proof}
We start again with one of the borders consisting of three elements (such a border exists according to Lemma 3).
First we list all decompositions of a 15-border into three or more elements. If a decomposition appears for the second time, it is put in brackets:
\begin{itemize}
	\item corner: 4, rest of border: 11 $\Rightarrow$ 4+3+8 or 4+5+6 or 4+6+5 or 4+7+4
	\item corner: 5, rest of border: 10 $\Rightarrow$ 5+3+7 or 5+4+6 (or 5+6+4)
	\item corner: 6, rest of border: 9 $\Rightarrow$ 6+3+6 (or 6+4+5 or 6+5+4)
	\item corner: 7, rest of border: 8 $\Rightarrow$ (7+3+5)
	\item corner: 8, rest of border: 7 $\Rightarrow$ (8+3+4)
\end{itemize}
But (because of Lemma \ref{lem64}) 4+6+5 and 4+7+4 are impossible, and (because of Lemma \ref{lemx3y}) 5+3+7 and 6+3+6 are impossible, and (because of Lemma \ref{lem43x}) 4+3+8 is impossible. So  we have two options left: 4+5+6 and 5+4+6.

First we try to start with 4+5+6:
	    
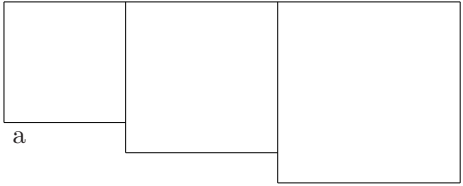
\begin{center}
			\begin{pspicture}(0,0)(15,6)
				\rput(0.5,1.5){a}
				\psline(0,6)(15,6)
				\psline(15,6)(15,0)
				\psline(15,0)(9,0)
				\psline(9,0)(9,6)
				\psline(9,1)(4,1)
				\psline(4,1)(4,6)
				\psline(4,2)(0,2)
				\psline(0,2)(0,6)
			\end{pspicture}
			\end{center}
In position $a$ we must have $a=3$ to avoid overlap, but there is no decomposition of a 15-border starting with 4+3. Thus 4+5+6 can't be completed to a squared square.

Next we try to start with 5+4+6:
	    
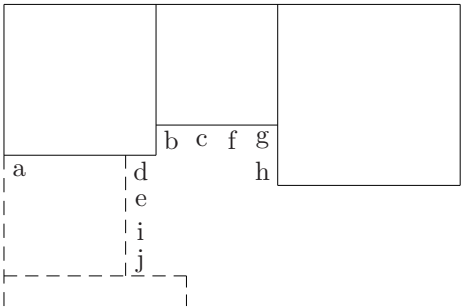
\begin{center}
			\begin{pspicture}(0,-4)(15,6)
				\rput(0.5,0.5){a}
				\rput(5.5,1.5){b}
				\rput(6.5,1.5){c}
				\rput(4.5,0.5){d}
				\rput(4.5,-0.5){e}
				\rput(7.5,1.5){f}
				\rput(8.5,1.5){g}
				\rput(8.5,0.5){h}
				\rput(4.5,-1.5){i}
				\rput(4.5,-2.5){j}
				\psline(0,6)(15,6)
				\psline(15,6)(15,0)
				\psline(15,0)(9,0)
				\psline(9,0)(9,6)
				\psline(9,2)(5,2)
				\psline(5,1)(5,6)
				\psline(5,1)(0,1)
				\psline(0,1)(0,6)
				\psline[linestyle=dashed](0,1)(0,-4)
				\psline[linestyle=dashed](0,-3)(6,-3)
				\psline[linestyle=dashed](6,-3)(6,-4)
				\psline[linestyle=dashed](4,1)(4,-3)
			\end{pspicture}
			\end{center}
Here $a$ can't reach the bottom (otherwise $b=c=1$), so we must have 5+4+6 on the left border (dashed lines).
Now $b\leq 2$ (otherwise $d=e=1$). But if $b=2$ then $f=1$ and $g$ can't be filled. So we must have $b=1$. Now if $cf=2$ then $g=1$ and $h$ can't be filled, so we must have $cfg=3$, and thus $de=2$, thus $i=1$ and $j$ can't be filled.
\end{proof}

\section{Size 16x16}
\label{Sect4}

\begin{lem}\label{lemx4y}It is not possible to complete a border decomposed $x+4+y$ to a full squared square if $x\geq 6$ and $y\geq 6$.\end{lem}
\begin{proof}
Such a decomposition would look like this:
	    
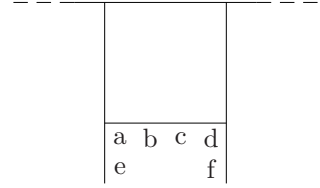
\begin{center}
			\begin{pspicture}(-2,0)(8,6)
				\rput(1.5,1.5){a}
				\rput(2.5,1.5){b}
				\rput(3.5,1.5){c}
				\rput(4.5,1.5){d}
				\rput(1.5,0.5){e}
				\rput(4.5,0.5){f}
				\psline(1,0)(1,6)
				\psline(0,6)(6,6)
				\psline(5,6)(5,0)
				\psline(1,2)(5,2)
				\psline[linestyle=dashed](0,6)(-2,6)
				\psline[linestyle=dashed](6,6)(8,6)
			\end{pspicture}
			\end{center}
But $a=4$ is not allowed. If $a=3$ we would get $d=f=1$ which is not allowed. If $a=2$ we would get $c=1$ and thus $d=1$, which is not allowed. Finally, if $a=1$ we would have $b\geq 2$ and thus $e=1$ which is not allowed.
\end{proof}

\begin{lem}\label{lem16-3}In size 16, no border can be decomposed into \emph{exactly three} elements.\end{lem}
\begin{proof}
In size 16x16, a border can be decomposed into four elements: 4+3+4+5 and 4+3+5+4. Besides, there are many possibilities to decompose a 16-border into \emph{three} elements -- if a decomposition appears for the second time, it is put in brackets:
\begin{itemize}
	\item corner 4, rest 12 $\Rightarrow$ 4+3+9 or 4+5+7 or 4+7+5 or 4+8+4
	\item corner 5, rest 11 $\Rightarrow$ 5+3+8 or 5+4+7 or 5+6+5 (or 5+7+4)
	\item corner 6, rest 10 $\Rightarrow$ 6+3+7 or 6+4+6
	\item corner 7, rest  9 $\Rightarrow$ (7+3+6 or 7+4+5 or 7+5+4)
	\item corner 8, rest  8 $\Rightarrow$ (8+3+5)
	\item corner 9, rest  7 $\Rightarrow$ (9+3+4)	
\end{itemize}
Lemma \ref{lemx3y} cancels out 5+3+8 and 6+3+7. Lemma \ref{lem64} cancels out 4+7+5 and 4+8+4. Lemma \ref{lemx4y} cancels out 6+4+6. Lemma \ref{lem43x} cancels out 4+3+9. This leaves the following three possible decompositions into \emph{three} elements: 4+5+7, 5+4+7 and 5+6+5.

First we start with 4+5+7:
	    
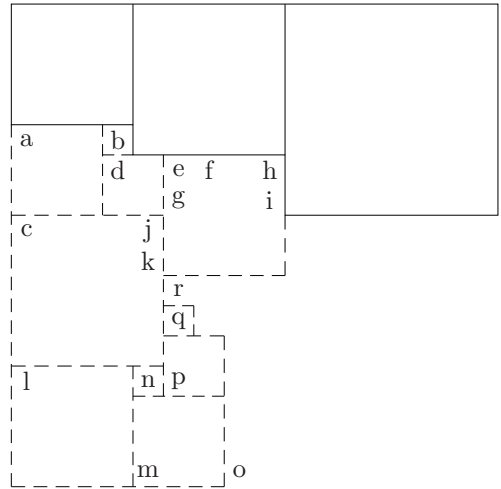
\begin{center}
			\begin{pspicture}(0,-2)(16,14)
				\rput(0.5,9.5){a}
				\rput(3.5,9.5){b}
				\rput(0.5,6.5){c}
				\rput(3.5,8.5){d}
				\rput(5.5,8.5){e}
				\rput(6.5,8.5){f}
				\rput(5.5,7.5){g}
				\rput(8.5,8.5){h}
				\rput(8.5,7.5){i}
				\rput(4.5,6.5){j}
				\rput(4.5,5.5){k}
				\rput(0.5,1.5){l}
				\rput(4.5,-1.5){m}
				\rput(4.5,1.5){n}
				\rput(7.5,-1.5){o}
				\rput(5.5,1.5){p}
				\rput(5.5,3.5){q}
				\rput(5.5,4.5){r}
				\pspolygon(0,10)(0,14)(16,14)(16,7)(9,7)(9,9)(4,9)(4,10)(0,10)
				\psline(4,10)(4,14)
				\psline(9,9)(9,14)
				\psline[linestyle=dashed](0,-2)(0,10)
				\psline[linestyle=dashed](0,2)(5,2)
				\psline[linestyle=dashed](0,7)(5,7)
				\psline[linestyle=dashed](3,9)(4,9)
				\psline[linestyle=dashed](3,7)(3,10)
				\psline[linestyle=dashed](4,-2)(4,2)
				\psline[linestyle=dashed](4,1)(7,1)
				\psline[linestyle=dashed](7,3)(7,-2)
				\psline[linestyle=dashed](7,3)(5,3)
				\psline[linestyle=dashed](5,1)(5,9)
				\psline[linestyle=dashed](5,5)(9,5)
				\psline[linestyle=dashed](9,5)(9,7)
				\psline[linestyle=dashed](0,-2)(7,-2)
				\psline[linestyle=dashed](5,4)(6,4)
				\psline[linestyle=dashed](6,4)(6,3)
			\end{pspicture}
			\end{center}
Because $a$ is at the border it must be $a\geq 3$. Because $a\neq 4$ we have $a=3$ and $b=1$. Because $c\geq 4$ we have $d=2$. If we had $e=1$ then we also had $f\geq 2$ and $g=1$ which is not allowed. Also $e=2$ is not allowed. If we had $e=3$ then we had $h=i=1$ which is not allowed. Thus we can conclude that $e=4$ (dashed lines). Now if we had $c=4$ then $j=k=1$ which is not allowed. Thus we have $c=5$ and $l=4$, thus $m=3$ and $n=1$. Now because $o\geq 4$ we have $p=2$. Thus $q=1$, thus $r$ can't be filled.

Second we start with 5+4+7:
	    
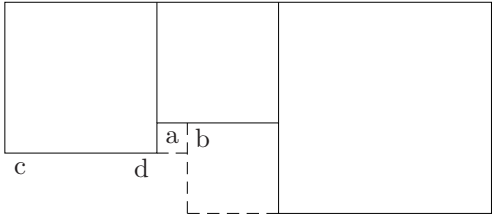
\begin{center}
			\begin{pspicture}(0,0)(16,7)
				\rput(5.5,2.5){a}
				\rput(6.5,2.5){b}
				\rput(0.5,1.5){c}
				\rput(4.5,1.5){d}
				\pspolygon(0,2)(5,2)(5,3)(9,3)(9,0)(16,0)(16,7)(0,7)(0,2)
				\psline(5,3)(5,7)
				\psline(9,3)(9,7)
				\psline[linestyle=dashed](5,2)(6,2)
				\psline[linestyle=dashed](6,0)(6,3)
				\psline[linestyle=dashed](6,0)(9,0)
			\end{pspicture}
			\end{center}
Here we must have $a=1$. From this follows $b=3$ (dashed lines). Then we have two possible ways to go on: (I) $c=4$ and $d=2$ or (II) $c=6$.
Try (I):
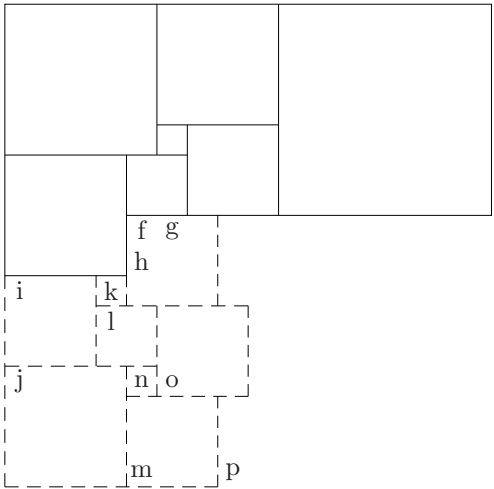
  	  \begin{center}
			\begin{pspicture}(0,0)(16,16)
				\pspolygon(0,7)(4,7)(4,9)(16,9)(16,16)(0,16)(0,7)
				\psline(0,11)(6,11)
				\psline(5,12)(9,12)
				\psline(4,9)(4,11)
				\psline(6,9)(6,12)
				\psline(9,9)(9,16)
				\psline(5,11)(5,16)
				\psline[linestyle=dashed](0,0)(0,7)
				\psline[linestyle=dashed](0,0)(7,0)
				\psline[linestyle=dashed](4,0)(4,4)
				\psline[linestyle=dashed](7,0)(7,3)
				\psline[linestyle=dashed](4,3)(8,3)
				\psline[linestyle=dashed](5,3)(5,6)
				\psline[linestyle=dashed](8,3)(8,6)
				\psline[linestyle=dashed](0,4)(5,4)
				\psline[linestyle=dashed](3,4)(3,7)
				\psline[linestyle=dashed](3,6)(8,6)
				\psline[linestyle=dashed](4,6)(4,7)
				\psline[linestyle=dashed](7,6)(7,9)
				\rput(4.5,8.5){f}
				\rput(5.5,8.5){g}
				\rput(4.5,7.5){h}
				\rput(0.5,6.5){i}
				\rput(0.5,3.5){j}
				\rput(3.5,6.5){k}
				\rput(3.5,5.5){l}
				\rput(4.5,0.5){m}
				\rput(4.5,3.5){n}
				\rput(5.5,3.5){o}
				\rput(7.5,0.5){p}
			\end{pspicture}
			\end{center}
If we had $f=1$, then $g\geq 2$ and $h$ couldn't be filled. But $f=2$ is not allowed, so we have $f\geq 3$. From this we get $i=3$, thus $j=4$. Because we must have $k=1$ and $l=2$ to fill this gap, it follows $m\leq 4$, thus $m=3$, thus $n=1$, $f=3$ and $o=3$. But now $p$ can't be filled.

Try (II):
	    
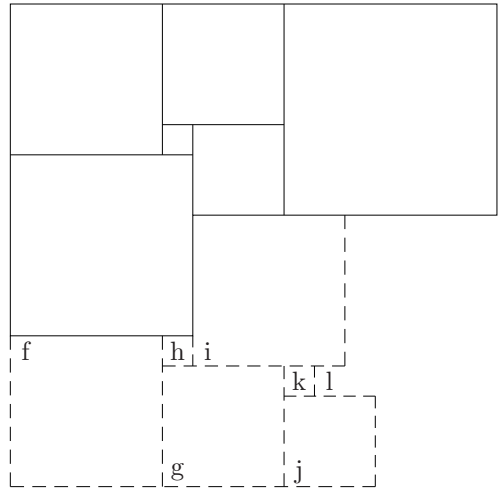
\begin{center}
			\begin{pspicture}(0,0)(16,16)
				\pspolygon(0,5)(6,5)(6,9)(16,9)(16,16)(0,16)(0,5)
				\psline(0,11)(6,11)
				\psline(6,9)(6,12)
				\psline(5,12)(9,12)
				\psline(9,9)(9,16)
				\psline(5,11)(5,16)
				\psline[linestyle=dashed](0,0)(0,5)
				\psline[linestyle=dashed](0,0)(12,0)
				\psline[linestyle=dashed](5,0)(5,5)
				\psline[linestyle=dashed](9,0)(9,4)
				\psline[linestyle=dashed](6,4)(6,5)
				\psline[linestyle=dashed](10,3)(10,4)
				\psline[linestyle=dashed](11,4)(11,9)
				\psline[linestyle=dashed](5,4)(11,4)
				\psline[linestyle=dashed](9,3)(12,3)
				\psline[linestyle=dashed](12,3)(12,0)
				\rput(0.5,4.5){f}
				\rput(5.5,0.5){g}
				\rput(5.5,4.5){h}
				\rput(6.5,4.5){i}
				\rput(9.5,0.5){j}
				\rput(9.5,3.5){k}
				\rput(10.5,3.5){l}
			\end{pspicture}
			\end{center}
Here we must have $f=5$ until the bottom. Because there is no 16-dissection starting with 5+3+\ldots, it follows $g\neq 3$, so we must have $g=4$ and $h=1$.
From this follows $i=5$, thus $j=3$, $k=1$, and $l$ can't be filled.

Third and finally we start with 5+6+5:
	    
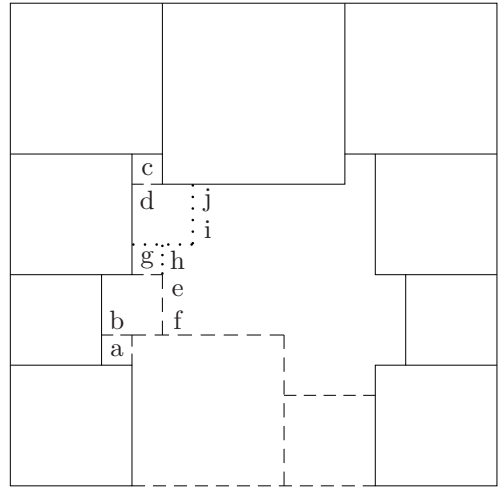
\begin{center}
			\begin{pspicture}(0,0)(16,16)
				\rput(3.5,4.5){a}
				\rput(3.5,5.5){b}
				\rput(4.5,10.5){c}
				\rput(4.5,9.5){d}
				\rput(5.5,6.5){e}
				\rput(5.5,5.5){f}
				\rput(4.5,7.5){g}
				\rput(5.5,7.5){h}
				\rput(6.5,8.5){i}
				\rput(6.5,9.5){j}
				\pspolygon(0,0)(4,0)(4,4)(3,4)(3,7)(4,7)(4,11)(5,11)(5,10)(11,10)(11,11)(12,11)(12,7)(13,7)(13,4)(12,4)(12,0)(16,0)(16,16)(0,16)(0,0)
				\psline(0,4)(3,4)
				\psline(0,7)(3,7)
				\psline(0,11)(4,11)
				\psline(5,16)(5,11)
				\psline(11,16)(11,11)
				\psline(12,11)(16,11)
				\psline(13,7)(16,7)
				\psline(13,4)(16,4)
				\psline[linestyle=dashed](4,0)(12,0)
				\psline[linestyle=dashed](9,0)(9,5)
				\psline[linestyle=dashed](9,5)(4,5)
				\psline[linestyle=dashed](4,5)(4,4)
				\psline[linestyle=dashed](9,3)(12,3)
				\psline[linestyle=dashed](3,5)(4,5)
				\psline[linestyle=dashed](5,5)(5,7)
				\psline[linestyle=dashed](5,7)(4,7)
				\psline[linestyle=dashed](4,10)(5,10)
				\psline[linestyle=dotted,linewidth=1pt](4,8)(6,8)
				\psline[linestyle=dotted,linewidth=1pt](6,8)(6,10)
				\psline[linestyle=dotted,linewidth=1pt](5,8)(5,7)
			\end{pspicture}
			\end{center}
The left and right border must both contain at least 3 elements, because if we put a 11x11 on one side it would overlap the 6x6 from the top border.
They have be (top to bottom) 5+4+3+4, the only possible remaining 16-dissection with a 5x5 in the corner (it is the only one because 5+6+5 would overlap). This is already shown in the picture above.
On the bottom border we can't put just one big 9x9 square, because we couldnt' fill the gaps to the left and right. So we must put 4+5+3+4 (or, symmetrically, 4+3+5+4) onto the bottom border. Because we don't distinguish symmetrical versions we can assume that we have 4+5+3+4 (dashed lines).
We must have $a=1$, $b=2$ and $c=1$ (dashed lines). If we had $d=3$, then $e=f=1$. So we must have $d=2$ and $g=1$ (dotted lines), thus $h\geq 2$, thus $i=1$ and $j$ can't be filled.

Because this was the last possible way to dissect a 16-border into three elements we have finished the proof of Lemma \ref{lem16-3}.
\end{proof}

\begin{lem}\label{lem16-4}Every squared square with size 16 has four elements on every border.\end{lem}
\begin{proof}
There are two possible dissections into four elements: 4+3+4+5 and 4+3+5+4. From Lemma \ref{lem16-3} we know that we can't decompose a 16-border into three elements. But from Lemma \ref{lem3} we know that there must have at least 6 elements along the border. And because it is not possible to decompose a 16-border into five or more elements, every 16x16-square must contain \emph{at least one} border dissected into exactly four elements.

If we start with 4+3+4+5, then we can't have a 12x12 square on the left border until the bottom, because it would overlap the existing 5x5 square. We also can't have a 11x11 square on the right border until the bottom, because the gap between this and the upper border elements can't be filled.

If we start with 4+3+5+4, then we can't have a 12x12 square on the left or right border until the bottom, because it would overlap the existing 5x5 square.

So we can assume that we have four elements on the top border, and we deduced that we must have four elements on the left and right border too. Thus follows we have at least three elements on the bottom border too (because we have 4x4 or 5x5 in both bottom corners, and there is no remaining dissection of the 16-border into three elements).

Thus we have proven that every 16x16 square has four elements on \emph{every} border.
\end{proof}

\begin{lem}\label{uni16}Besides the squared square shown on page \pageref{fig1116} (and the mirrored version of it), there is no other squared square of size 16.\end{lem}
\begin{proof}
From Lemma \ref{lem16-4} we deduce that on every border of a 16x16 square there is either 4+3+4+5 or 4+3+5+4. So there are three possible arrangements of \emph{corners}:
	    
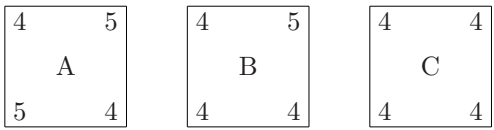
\begin{center}
			\begin{pspicture}(0,0)(16,4)
				\rput(0.5,0.5){5}
				\rput(0.5,3.5){4}
				\rput(2,2){A}
				\rput(3.5,0.5){4}
				\rput(3.5,3.5){5}
				\psline(0,0)(0,4)
				\psline(0,4)(4,4)
				\psline(4,4)(4,0)
				\psline(4,0)(0,0)
				\rput(6.5,0.5){4}
				\rput(6.5,3.5){4}
				\rput(8,2){B}
				\rput(9.5,0.5){4}
				\rput(9.5,3.5){5}
				\psline(6,0)(6,4)
				\psline(6,4)(10,4)
				\psline(10,4)(10,0)
				\psline(10,0)(6,0)
				\rput(12.5,0.5){4}
				\rput(12.5,3.5){4}
				\rput(14,2){C}
				\rput(15.5,0.5){4}
				\rput(15.5,3.5){4}
				\psline(12,0)(12,4)
				\psline(12,4)(16,4)
				\psline(16,4)(16,0)
				\psline(16,0)(12,0)
			\end{pspicture}
			\end{center}
First we observe that there is only one way (A1) to complete the borders between the corners A (see below).
If there is a 4x4 in the corner, then it is not possible that \emph{both} neighbors are 5x5 (would lead to an overlap). 
Because of this and because we don't distinguish symmetric decompositions, there are only two ways (B1) and (B2) to complete the corners B and only one way (C1) to complete the corners C:
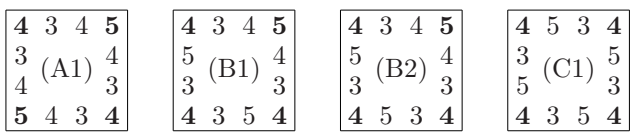
	    \begin{center}
			\begin{pspicture}(0,0)(20.5,4)
				\rput(0.5,3.5){\textbf{4}}
				\rput(1.5,3.5){3}
				\rput(2.5,3.5){4}
				\rput(3.5,3.5){\textbf{5}}
				\rput(0.5,2.5){3}
				\rput(3.5,2.5){4}
				\rput(0.5,1.5){4}
				\rput(3.5,1.5){3}
				\rput(0.5,0.5){\textbf{5}}
				\rput(1.5,0.5){4}
				\rput(2.5,0.5){3}
				\rput(3.5,0.5){\textbf{4}}
				\rput(2,2){(A1)}
				\psline(0,0)(0,4)
				\psline(0,4)(4,4)
				\psline(4,4)(4,0)
				\psline(4,0)(0,0)
				\rput(6,3.5){\textbf{4}}
				\rput(7,3.5){3}
				\rput(8,3.5){4}
				\rput(9,3.5){\textbf{5}}
				\rput(6,2.5){5}
				\rput(9,2.5){4}
				\rput(6,1.5){3}
				\rput(9,1.5){3}
				\rput(6,0.5){\textbf{4}}
				\rput(7,0.5){3}
				\rput(8,0.5){5}
				\rput(9,0.5){\textbf{4}}
				\rput(7.5,2){(B1)}
				\psline(5.5,0)(5.5,4)
				\psline(5.5,4)(9.5,4)
				\psline(9.5,4)(9.5,0)
				\psline(9.5,0)(5.5,0)
				\rput(11.5,3.5){\textbf{4}}
				\rput(12.5,3.5){3}
				\rput(13.5,3.5){4}
				\rput(14.5,3.5){\textbf{5}}
				\rput(11.5,2.5){5}
				\rput(14.5,2.5){4}
				\rput(11.5,1.5){3}
				\rput(14.5,1.5){3}
				\rput(11.5,0.5){\textbf{4}}
				\rput(12.5,0.5){5}
				\rput(13.5,0.5){3}
				\rput(14.5,0.5){\textbf{4}}
				\rput(13,2){(B2)}
				\psline(11,0)(11,4)
				\psline(11,4)(15,4)
				\psline(15,4)(15,0)
				\psline(15,0)(11,0)
				\rput(17,3.5){\textbf{4}}
				\rput(18,3.5){5}
				\rput(19,3.5){3}
				\rput(20,3.5){\textbf{4}}
				\rput(17,2.5){3}
				\rput(20,2.5){5}
				\rput(17,1.5){5}
				\rput(20,1.5){3}
				\rput(17,0.5){\textbf{4}}
				\rput(18,0.5){3}
				\rput(19,0.5){5}
				\rput(20,0.5){\textbf{4}}
				\rput(18.5,2){(C1)}
				\psline(16.5,0)(16.5,4)
				\psline(16.5,4)(20.5,4)
				\psline(20.5,4)(20.5,0)
				\psline(20.5,0)(16.5,0)
			\end{pspicture}
			\end{center}
Now we will try to fill the center parts of these four dissections and we'll observe that only (C1) leads to a (unique) squared square.

The remaining center part of (A1):
	    
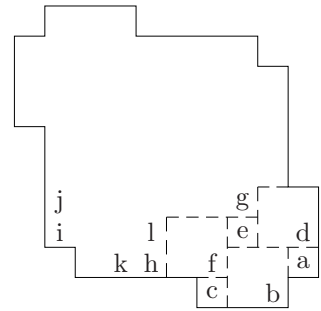
\begin{center}
			\begin{pspicture}(0,0)(10,10)
				\pspolygon(2,1)(6,1)(6,0)(9,0)(9,1)(10,1)(10,4)(9,4)(9,8)(8,8)(8,9)(4,9)(4,10)(1,10)(1,9)(0,9)(0,6)(1,6)(1,2)(2,2)(2,1)
				\rput(9.5,1.5){a}
				\rput(8.5,0.5){b}
				\rput(6.5,0.5){c}
				\rput(9.5,2.5){d}
				\rput(7.5,2.5){e}
				\rput(6.5,1.5){f}
				\rput(7.5,3.5){g}
				\rput(4.5,1.5){h}
				\rput(1.5,2.5){i}
				\rput(1.5,3.5){j}
				\rput(3.5,1.5){k}
				\rput(4.5,2.5){l}
				\psline[linestyle=dashed](9,1)(9,2)
				\psline[linestyle=dashed](10,2)(7,2)
				\psline[linestyle=dashed](7,0)(7,3)
				\psline[linestyle=dashed](7,1)(6,1)
				\psline[linestyle=dashed](8,2)(8,4)
				\psline[linestyle=dashed](9,4)(8,4)
				\psline[linestyle=dashed](5,1)(5,3)
				\psline[linestyle=dashed](5,3)(8,3)
			\end{pspicture}
			\end{center}
We have $b\leq 3$, but if $b=3$ then $a=d=1$ which is not allowed. So either $b=2$ and $a=1$ or symmetrically $b=1$ and $a=2$. We can assume $a=1$ and $b=2$ (dashed lines). It follows $d=2$ and $c=1$. Because $f\geq 2$ we must have $e=1$. But if $f\geq 3$ then $g=e=1$, so we must have $f=2$ (dashed lines). We must have $h\leq 3$. But if we had $h=3$ we would get $i=j=1$, which is not allowed. $h=2$ is not allowed because of $f=2$, so we must have $h=1$. But then $k=2$ and thus $l=1$, which is not allowed. This means that the center part of (A1) can't be filled.
			
The remaining center part of (B1):
	    
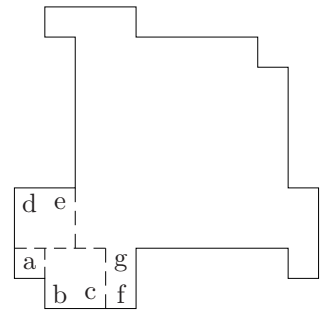
\begin{center}
			\begin{pspicture}(0,0)(10,10)
				\pspolygon(1,0)(4,0)(4,2)(9,2)(9,1)(10,1)(10,4)(9,4)(9,8)(8,8)(8,9)(4,9)(4,10)(1,10)(1,9)(2,9)(2,4)(0,4)(0,1)(1,1)(1,0)
				\rput(0.5,1.5){a}
				\rput(1.5,0.5){b}
				\rput(2.5,0.5){c}
				\rput(0.5,3.5){d}
				\rput(1.5,3.5){e}
				\rput(3.5,0.5){f}
				\rput(3.5,1.5){g}
				\psline[linestyle=dashed](0,2)(3,2)
				\psline[linestyle=dashed](3,2)(3,0)
				\psline[linestyle=dashed](1,1)(1,2)
				\psline[linestyle=dashed](2,2)(2,4)
			\end{pspicture}
			\end{center}
We must have $a\leq 3$, but if $a=3$ then $b=c=1$ and if $a=2$ then $d=e=1$, thus we must have $a=1$. It follows $de=2$, thus $bc=2$ (dashed lines). But now $f=g=1$, which is not allowed. 
			
Center part of (B2):
	    
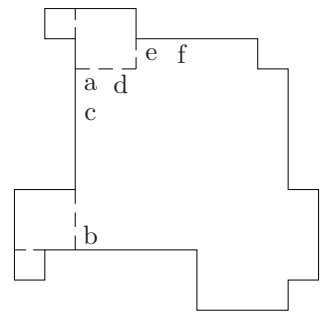
\begin{center}
			\begin{pspicture}(0,0)(10,10)
				\pspolygon(0,1)(1,1)(1,2)(6,2)(6,0)(9,0)(9,1)(10,1)(10,4)(9,4)(9,8)(8,8)(8,9)(4,9)(4,10)(1,10)(1,9)(2,9)(2,4)(0,4)(0,1)
				\psline[linestyle=dashed](0,2)(1,2)
				\psline[linestyle=dashed](2,2)(2,4)
				\psline[linestyle=dashed](2,9)(2,10)
				\psline[linestyle=dashed](2,8)(4,8)
				\psline[linestyle=dashed](4,8)(4,9)
				\rput(2.5,7.5){a}
				\rput(2.5,2.5){b}
				\rput(2.5,6.5){c}
				\rput(3.5,7.5){d}
				\rput(4.5,8.5){e}
				\rput(5.5,8.5){f}
			\end{pspicture}
			\end{center}
We have to fill the left border from $a$ to $b$. First we observe that $a\neq 1$, because otherwise we had $c\geq 2$ and $d=1$. Analogously we have $b\neq 1$. So we have $a\geq 3$ and $b\geq 3$. But $a=b=3$ is not allowed, thus $ab=6$. But now $e=f=1$ which isn't allowed either.
			
Center part of (C1):
	    
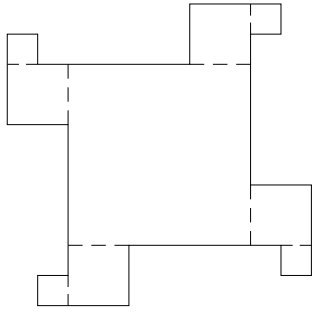
\begin{center}
			\begin{pspicture}(0,0)(10,10)
				\pspolygon(1,0)(4,0)(4,2)(9,2)(9,1)(10,1)(10,4)(8,4)(8,9)(9,9)(9,10)(6,10)(6,8)(1,8)(1,9)(0,9)(0,6)(2,6)(2,1)(1,1)(1,0)
				\psline[linestyle=dashed](2,0)(2,1)
				\psline[linestyle=dashed](2,2)(4,2)
				\psline[linestyle=dashed](2,8)(2,6)
				\psline[linestyle=dashed](1,8)(0,8)
				\psline[linestyle=dashed](6,8)(8,8)
				\psline[linestyle=dashed](8,10)(8,9)
				\psline[linestyle=dashed](8,4)(8,2)
				\psline[linestyle=dashed](9,2)(10,2)
			\end{pspicture}
			\end{center}
The outmost squares have to be 1x1, thus the neighboring squares have to be 2x2. The remaining 6x6 in the middle can't be furtherly dissected into squares because of Lemma \ref{least11}. Thus this is the only possible dissection of the center part of (C1) into squares. The whole 16x16 solution is shown in Figure 1 on page \pageref{fig1116}.
\end{proof}

\section{Size 17x17}
\label{Sect5}
\begin{lem}\label{no43535}The following pattern can't be completed: 5x5 in the corner with both neighbors 3x3 on the borders, and 4x4 squares neighboring the 3x3 on both borders (see below).\end{lem}
\begin{proof}
The described situation would look like this:
	    
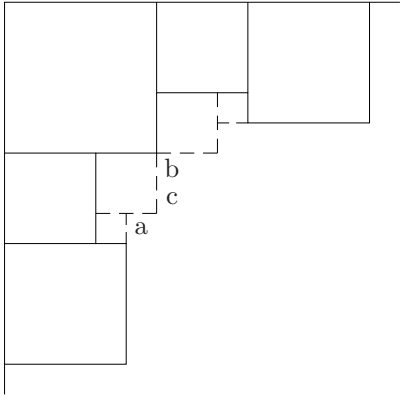
\begin{center}
			\begin{pspicture}(0,0)(13,13)
				\pspolygon(0,1)(4,1)(4,5)(3,5)(3,8)(5,8)(5,10)(8,10)(8,9)(12,9)(12,13)(0,13)(0,1)
				\psline(0,0)(0,1)
				\psline(12,13)(13,13)
				\psline(0,5)(3,5)
				\psline(0,8)(3,8)
				\psline(5,13)(5,10)
				\psline(8,13)(8,10)
				\psline[linestyle=dashed](3,6)(5,6)
				\psline[linestyle=dashed](4,5)(4,6)
				\psline[linestyle=dashed](5,6)(5,8)
				\psline[linestyle=dashed](5,8)(7,8)
				\psline[linestyle=dashed](7,8)(7,10)
				\psline[linestyle=dashed](7,9)(8,9)
				\rput(4.5,5.5){a}
				\rput(5.5,7.5){b}
				\rput(5.5,6.5){c}
			\end{pspicture}
			\end{center}
Here we already filled the gaps with dashed lines, because this is the only way they can be filled. But now we have $a\geq 2$, thus $b=1$, but then $c$ can't be filled.
\end{proof}

\begin{lem}\label{no17}There is no squared square of size 17x17.\end{lem}
\begin{proof}First we list all possible 17-dissections into three or more elements which have minimum size 4x4 in the corners and minimum size 3x3 at the border. We list them in ascending numerical order:
\noindent{\small\begin{verbatim}(1) 4+3+4+6 (2) 4+3+6+4  (3) 4+3+10   (4) 4+5+3+5
(5) 4+5+8   (6) 4+6+3+4  (7) 4+6+7    (8) 4+7+6
(9) 4+8+5  (10) 4+9+4   (11) 5+3+4+5 (12) 5+3+5+4
(13) 5+3+9 (14) 5+4+3+5 (15) 5+4+8   (16) 5+7+5
(17) 5+8+4 (18) 6+3+8   (19) 6+4+3+4 (20) 6+4+7
(21) 6+5+6 (22) 6+7+4   (23) 7+3+7   (24) 7+4+6
(25) 7+6+4 (26) 8+3+6   (27) 8+4+5   (28) 8+5+4
(29) 9+3+5 (30) 10+3+4\end{verbatim}}
 
Of these 30, most appear twice (in both directions). The remaining ones:
\noindent{\small\begin{verbatim}(1) 4+3+4+6 (2) 4+3+6+4  (3) 4+3+10 (4) 4+5+3+5
(5) 4+5+8   (7) 4+6+7    (8) 4+7+6  (9) 4+8+5
(10) 4+9+4 (11) 5+3+4+5 (13) 5+3+9 (15) 5+4+8
(16) 5+7+5 (18) 6+3+8   (20) 6+4+7 (21) 6+5+6
(23) 7+3+7\end{verbatim}}
 
Now we also apply Lemma \ref{lemx3y}, \ref{lem64}, \ref{lem43x} and \ref{lemx4y} to cancel some out. The remaining dissections:
\noindent{\small\begin{verbatim}(1) 4+3+4+6 (5) 4+5+8 (11) 5+3+4+5 (15) 5+4+8
(16) 5+7+5 (21) 6+5+6\end{verbatim}}
 
We need to prove that none of these six 17-dissections can be completed to a squared square. First we do (5), (15), (16). Then (11). Then (1) and (21) together.
 
\textbf{(5): 4+5+8} If we start with 4+5+8 on the top border, we must extend the 4x4 in the square with the 4+3+4+6-dissection on the left border, because both 4+13 and 4+5+8 would overlap. Thus we get:
	    
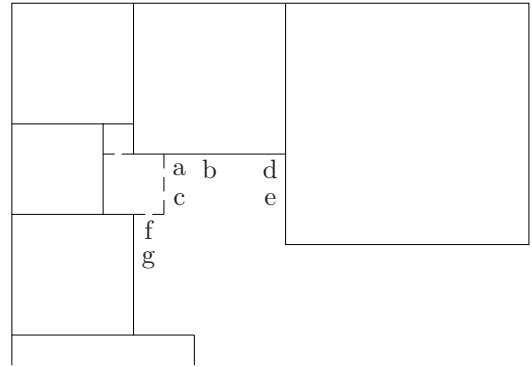
\begin{center}
			\begin{pspicture}(0,5)(17,17)
				\pspolygon(0,6)(4,6)(4,10)(3,10)(3,13)(4,13)(4,12)(9,12)(9,9)(17,9)(17,17)(0,17)(0,6)
				\psline[linestyle=dashed](3,12)(4,12)
				\psline[linestyle=dashed](5,12)(5,10)
				\psline[linestyle=dashed](4,10)(5,10)
				\psline(0,5)(0,6)
				\psline(4,6)(6,6)
				\psline(6,6)(6,5)
				\psline(0,10)(3,10)
				\psline(0,13)(3,13)
				\psline(4,17)(4,13)
				\psline(9,17)(9,12)
				\rput(5.5,11.5){a}
				\rput(6.5,11.5){b}
				\rput(5.5,10.5){c}
				\rput(8.5,11.5){d}
				\rput(8.5,10.5){e}
				\rput(4.5,9.5){f}
				\rput(4.5,8.5){g}
			\end{pspicture}
			\end{center}
So if we had $a=1$ then $b\geq 2$ and $c$ couldn't be filled. $a=2$ is not allowed. If we had $a=3$ then $d=e=1$. If we had $a=4$ then $f=g=1$, thus 4+5+8 doesn't lead to a squared sqare.

\textbf{(15): 5+4+8} If we start with 5+4+8, we must have the following situation:
	    
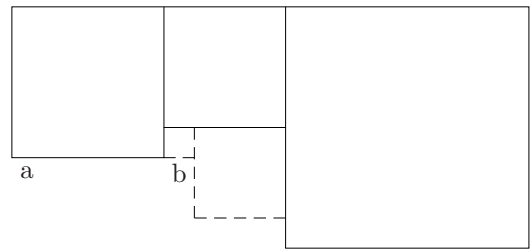
\begin{center}
			\begin{pspicture}(0,9)(17,17)
				\pspolygon(0,12)(5,12)(5,13)(9,13)(9,9)(17,9)(17,17)(0,17)(0,12)
				\psline[linestyle=dashed](5,12)(6,12)
				\psline[linestyle=dashed](6,13)(6,10)
				\psline[linestyle=dashed](6,10)(9,10)
				\psline(5,13)(5,17)
				\psline(9,13)(9,17)
				\rput(0.5,11.5){a}
				\rput(5.5,11.5){b}
			\end{pspicture}
			\end{center}
If we had $a=3$, then $b$ couldn't be filled (because only $b=2$ would be allowed, and here we couldn't fill the gap between $b$ and $a$). Thus we have two cases for the left border: (i)~5+4+8 or (ii)~5+4+3+5 (top to bottom), which means $a=4$.

Case (i): 5+4+8 on the left border:
	    
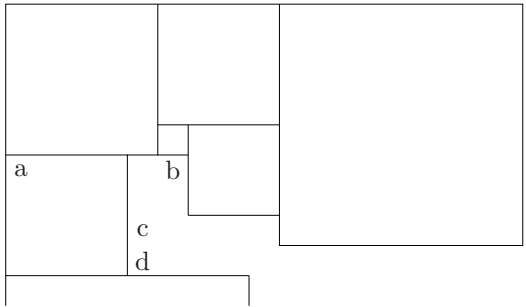
\begin{center}
			\begin{pspicture}(0,7)(17,17)
				\pspolygon(0,8)(4,8)(4,12)(5,12)(5,13)(9,13)(9,9)(17,9)(17,17)(0,17)(0,8)
				\psline(5,12)(6,12)
				\psline(6,13)(6,10)
				\psline(6,10)(9,10)
				\psline(5,13)(5,17)
				\psline(0,12)(4,12)
				\psline(9,13)(9,17)
				\psline(0,8)(0,7)
				\psline(4,8)(8,8)
				\psline(8,8)(8,7)
				\rput(0.5,11.5){a}
				\rput(5.5,11.5){b}
				\rput(4.5,9.5){c}
				\rput(4.5,8.5){d}
			\end{pspicture}
			\end{center}
Here we must have $b=2$, but now $c=1$ and $d$ can't be filled.

So it remains case (ii): 5+4+3+5 on the left border. Here, if the right border of the whole square would be 8+9, then we would have $5+x+9$ on the bottom border, which is not possible. Thus the right border must contain at least three elements and thus has to be 8+4+5.
But now (from symmetrical argumentation like above) we must have 5+3+4+5 on the bottom (left to right):
	    \begin{center}
			\begin{pspicture}(0,0)(4,4)
				\rput(0.5,3.5){\textbf{5}}
				\rput(2,3.5){4}
				\rput(3.5,3.5){\textbf{8}}
				\rput(0.5,2.5){4}
				\rput(3.5,2){4}
				\rput(0.5,1.5){3}
				\rput(0.5,0.5){\textbf{5}}
				\rput(1.5,0.5){3}
				\rput(2.5,0.5){4}
				\rput(3.5,0.5){\textbf{5}}
				\psline(0,0)(0,4)
				\psline(0,4)(4,4)
				\psline(4,4)(4,0)
				\psline(4,0)(0,0)
			\end{pspicture}
			\end{center}
But now Lemma \ref{no43535} applies to the bottom left corner and tells us that the center part of this can't be completed, thus 5+4+8 doesn't lead to a squared square.

\textbf{(16): 5+7+5} Now we start with 5+7+5 on the top border. To fill the left (and right) border, we have only two 17-dissections left which start with a corner-5: 5+7+5 or 5+3+4+5. But 5+7+5 on the left (or right) border would overlap with the 7x7 on the top border. So we must have 5+3+4+5 on the left (and right) border. And the 3x3 border squares must be on top of the 4x4 border squares, otherwise we would get a gap between the 4x4 on the left (right) border and the 7x7 on the top border. Thus we get:
	    
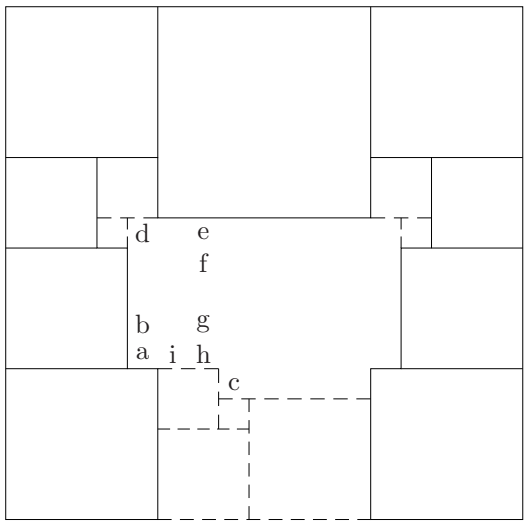
\begin{center}
			\begin{pspicture}(0,0)(17,17)
				\pspolygon(0,0)(5,0)(5,5)(4,5)(4,9)(3,9)(3,12)(5,12)(5,10)(12,10)(12,12)(14,12)(14,9)(13,9)(13,5)(12,5)(12,0)(17,0)(17,17)(0,17)(0,0)
				\psline[linestyle=dashed](3,10)(5,10)
				\psline[linestyle=dashed](4,10)(4,9)
				\psline[linestyle=dashed](12,10)(14,10)
				\psline[linestyle=dashed](13,10)(13,9)
				\psline[linestyle=dashed](5,3)(8,3)
				\psline[linestyle=dashed](8,0)(8,4)
				\psline[linestyle=dashed](12,4)(7,4)
				\psline[linestyle=dashed](7,3)(7,5)
				\psline[linestyle=dashed](7,5)(5,5)
				\psline[linestyle=dashed](5,0)(12,0)
				\psline(0,5)(4,5)
				\psline(0,9)(3,9)
				\psline(0,12)(3,12)
				\psline(5,12)(5,17)
				\psline(12,12)(12,17)
				\psline(14,12)(17,12)
				\psline(14,9)(17,9)
				\psline(13,5)(17,5)
				
				\rput(4.5,5.5){a}
				\rput(4.5,6.5){b}
				\rput(7.5,4.5){c}
				\rput(4.5,9.5){d}
				\rput(6.5,9.5){e}
				\rput(6.5,8.5){f}
				\rput(6.5,6.5){g}
				\rput(6.5,5.5){h}
				\rput(5.5,5.5){i}
			\end{pspicture}
			\end{center}
Now we can't have 5+7+5 on the bottom border (otherwise $a=b=1$), so we have 5+4+3+5 or 5+3+4+5 on the bottom. Because we don't distinguish symmetric solutions we can assume the dahed decomposition of the bottom border.
Next we already marked the 2x2 and 1x1 squares in dashed lines. Now we will show that $a$ can't be filled: If $a\geq 4$ then $c=1$ which isn't allowed. If $a=3$ then $d=2$ but then $e=f=1$ which is not allowed. If $a=2$ then $d=3$ but then $g=h=1$ which is not allowed. If $a=1$ then $b\geq 2$ but then $i=1$ which is not allowed. Thus $a$ can't be filled, thus no squared square can contain a 5+7+5 border.

\textbf{(11): 5+3+4+5} If we start with 5+3+4+5 on the top border, we need to complete the left border starting with a 5x5 square. The square $x$ can't reach the bottom, otherwise $x=17-5=12$ and $l=h=1$. So the left border has to be dissected into at least three elements, but the only possible 17-dissection with three or more elements is 5+3+4+5. Analogously we deduce that we must have 5+3+4+5 on \emph{every} border. The only choice we seem to have is the direction (on every border: 5+3+4+5 or 5+4+3+5). But Lemma \ref{no43535} says that in \emph{this} dissection it is not possible to have 3x3's on both sides around a corner, thus the only possible arrangement is the following (or the mirrored version):
	    
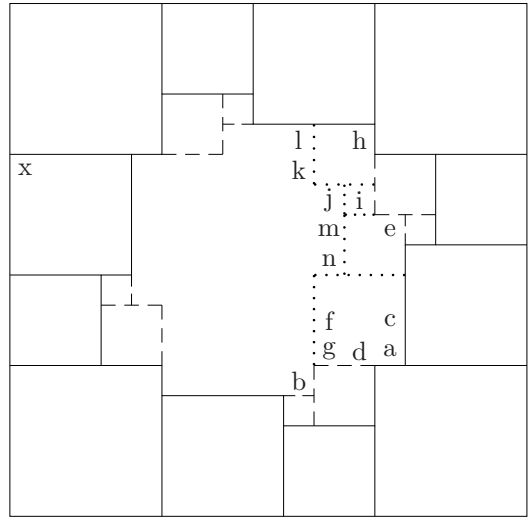
\begin{center}
			\begin{pspicture}(0,0)(17,17)
			  \rput(0.5,11.5){x}
				\pspolygon(0,0)(17,0)(17,17)(0,17)(0,0)
				\pspolygon(5,5)(5,4)(9,4)(9,3)(12,3)(12,5)(13,5)(13,9)(14,9)(14,12)(12,12)(12,13)(8,13)(8,14)(5,14)(5,12)(4,12)(4,8)(3,8)(3,5)(5,5)
				\psline[linestyle=dashed](12,10)(14,10)
				\psline[linestyle=dashed](13,10)(13,9)
				\psline[linestyle=dashed](12,10)(12,12)
				\psline[linestyle=dashed](9,4)(10,4)
				\psline[linestyle=dashed](10,3)(10,5)
				\psline[linestyle=dashed](10,5)(12,5)
				\psline[linestyle=dashed](4,8)(4,7)
				\psline[linestyle=dashed](3,7)(5,7)
				\psline[linestyle=dashed](5,7)(5,5)
				\psline[linestyle=dashed](5,12)(7,12)
				\psline[linestyle=dashed](7,12)(7,14)
				\psline[linestyle=dashed](7,13)(8,13)
				\psline(5,0)(5,4)
				\psline(9,0)(9,3)
				\psline(12,0)(12,3)
				\psline(13,5)(17,5)
				\psline(14,9)(17,9)
				\psline(14,12)(17,12)
				\psline(12,13)(12,17)
				\psline(8,14)(8,17)
				\psline(5,14)(5,17)
				\psline(0,12)(4,12)
				\psline(0,8)(3,8)
				\psline(0,5)(3,5)
				\rput(12.5,5.5){a}
				\rput(9.5,4.5){b}
				\rput(12.5,6.5){c}
				\rput(11.5,5.5){d}
				\rput(12.5,9.5){e}
				\rput(10.5,6.5){f}
				\rput(10.5,5.5){g}
				\rput(11.5,12.5){h}
				\rput(11.5,10.5){i}
				\rput(10.5,10.5){j}
				\rput(9.5,11.5){k}
				\rput(9.5,12.5){l}
				\rput(10.5,9.5){m}
				\rput(10.5,8.5){n}
				\psline[linestyle=dotted,linewidth=1pt](10,5)(10,8)
				\psline[linestyle=dotted,linewidth=1pt](10,8)(13,8)
				\psline[linestyle=dotted,linewidth=1pt](11,8)(11,10)
				\psline[linestyle=dotted,linewidth=1pt](11,10)(12,10)
				\psline[linestyle=dotted,linewidth=1pt](11,10)(11,11)
				\psline[linestyle=dotted,linewidth=1pt](12,11)(10,11)
				\psline[linestyle=dotted,linewidth=1pt](10,11)(10,13)
			\end{pspicture}
			\end{center}
Again, we already marked the 1x1 and 2x2 squares on the border of the remaining center part.
First we observe that $a\leq 3$, otherwise $b$ can't be filled. Then we observe that $a\neq 1$ because otherwise $c\geq 2$ and $d$ couldn't be filled. If we had $a=2$ we would get $e=3$, but then $f=g=1$ which isn't allowed. So the only remaining possibility is $a=3$ (and thus $e=2$) (dotted lines). Now if $h=3$ then $mn$ can't be filled. Thus $h=2$ and $i=1$ (dotted). But now we must have $j\geq 2$, but now $kl$ can't be filled. Thus 5+3+4+5 doesn't lead to a squared square either.

\textbf{(1): 4+3+4+6 and (21): 6+5+6} These are now the only remaining dissections of the 17-border into three or more elements. If we start with 6+5+6 on the top border, we can't have only two squares on the left (or right) border (because we couldn't fill the gap underneath the top 5x5 square). The same is true if we start with the 4+3+4+6 (because we couldn't fill the gap beneath the top center 4x4 square). But these two dissections can be combined in different ways, leading to different squares in the corners (4x4 or 6x6):
\begin{itemize}
\item If \emph{four} 6x6 squares are in the corners, the border-dissections must be 6+5+6 on every border (see I).
\item If \emph{three} 6x6 squares are in the corners, there must be 6+5+6 on two borders and 6+4+3+4 on two borders (see II).
\item If \emph{two} 6x6 squares are in the corners, there must be 6+4+3+4 on every border (see III).
\item \emph{one or no} 6x6 squares in the corners is impossible.
\end{itemize}
	    \begin{center}
			\begin{pspicture}(0,0)(16,4)
				\rput(0.5,3.5){\textbf{6}}
				\rput(2,3.5){5}
				\rput(3.5,3.5){\textbf{6}}
				\rput(0.5,2){5}
				\rput(3.5,2){5}
				\rput(0.5,0.5){\textbf{6}}
				\rput(2,0.5){5}
				\rput(3.5,0.5){\textbf{6}}
				\rput(2,2){(I)}
				\psline(0,0)(0,4)
				\psline(0,4)(4,4)
				\psline(4,4)(4,0)
				\psline(4,0)(0,0)
				
				\rput(6.5,3.5){\textbf{6}}
				\rput(8,3.5){5}
				\rput(9.5,3.5){\textbf{6}}
				\rput(6.5,2){5}
				\rput(9.5,2.5){4}
				\rput(9.5,1.5){3}
				\rput(6.5,0.5){\textbf{6}}
				\rput(7.5,0.5){4}
				\rput(8.5,0.5){3}
				\rput(9.5,0.5){\textbf{4}}
				\rput(8,2){(II)}
				\psline(6,0)(6,4)
				\psline(6,4)(10,4)
				\psline(10,4)(10,0)
				\psline(10,0)(6,0)
				
				\rput(12.5,3.5){\textbf{6}}
				\rput(13.5,3.5){4}
				\rput(14.5,3.5){3}
				\rput(15.5,3.5){\textbf{4}}
				\rput(12.5,2.5){4}
				\rput(15.5,2.5){3}
				\rput(12.5,1.5){3}
				\rput(15.5,1.5){4}
				\rput(12.5,0.5){\textbf{4}}
				\rput(13.5,0.5){3}
				\rput(14.5,0.5){4}
				\rput(15.5,0.5){\textbf{6}}
				\rput(14,2){(III)}
				\psline(12,0)(12,4)
				\psline(12,4)(16,4)
				\psline(16,4)(16,0)
				\psline(16,0)(12,0)
			\end{pspicture}
			\end{center}
The remaining center part of (I):
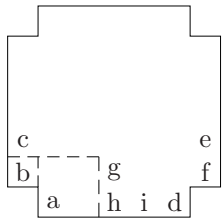
	    \begin{center}
			\begin{pspicture}(0,0)(7,7)
				\pspolygon(1,1)(1,0)(6,0)(6,1)(7,1)(7,6)(6,6)(6,7)(1,7)(1,6)(0,6)(0,1)(1,1)
				\psline[linestyle=dashed](0,2)(3,2)
				\psline[linestyle=dashed](3,0)(3,2)
				\psline[linestyle=dashed](1,1)(1,2)				
				\rput(1.5,0.5){a}
				\rput(0.5,1.5){b}
				\rput(0.5,2.5){c}
				\rput(5.5,0.5){d}
				\rput(6.5,2.5){e}
				\rput(6.5,1.5){f}
				\rput(3.5,1.5){g}
				\rput(3.5,0.5){h}
				\rput(4.5,0.5){i}
			\end{pspicture}
			\end{center}
We must have either $a=1$ or $b=1$. Thus we can assume $b=1$. But if $a\geq 3$ then $c$ can't be filled, thus we must have $a=2$ (dashed lines).
Now we must have $d\leq 2$, otherwise $e=f=1$. But if $d=2$ then $gh$ can't be filled. And if $d=1$, $i$ can't be filled.

The remaining center part of (II):
	    
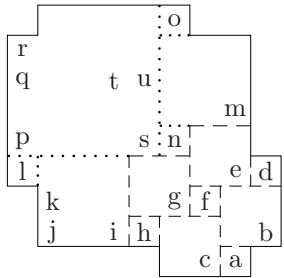
\begin{center}
			\begin{pspicture}(0,0)(9,9)
				\pspolygon(1,1)(5,1)(5,0)(8,0)(8,1)(9,1)(9,4)(8,4)(8,8)(6,8)(6,9)(1,9)(1,8)(0,8)(0,3)(1,3)(1,1)
				\psline[linestyle=dashed](4,1)(4,4)
				\psline[linestyle=dashed](6,4)(4,4)
				\psline[linestyle=dashed](6,2)(6,5)
				\psline[linestyle=dashed](8,5)(6,5)
				\psline[linestyle=dashed](6,3)(9,3)
				\psline[linestyle=dashed](8,3)(8,4)
				\psline[linestyle=dashed](7,0)(7,3)
				\psline[linestyle=dashed](7,1)(8,1)
				\psline[linestyle=dashed](4,2)(7,2)
				\psline[linestyle=dashed](5,1)(5,2)
				\rput(7.5,0.5){a}
				\rput(8.5,1.5){b}
				\rput(6.5,0.5){c}
				\rput(8.5,3.5){d}
				\rput(7.5,3.5){e}
				\rput(6.5,2.5){f}
				\rput(5.5,2.5){g}
				\rput(4.5,1.5){h}
				\rput(3.5,1.5){i}
				\rput(1.5,1.5){j}
				\rput(1.5,2.5){k}
				\rput(0.5,3.5){l}
				\rput(7.5,5.5){m}
				\rput(5.5,4.5){n}
				\rput(5.5,8.5){o}
				\psline[linestyle=dotted,linewidth=1pt](0,4)(4,4)
				\psline[linestyle=dotted,linewidth=1pt](1,3)(1,4)
				\psline[linestyle=dotted,linewidth=1pt](5,4)(5,9)
				\psline[linestyle=dotted,linewidth=1pt](5,5)(6,5)
				\psline[linestyle=dotted,linewidth=1pt](5,8)(6,8)
				\rput(0.5,4.5){p}
				\rput(0.5,6.5){q}
				\rput(0.5,7.5){r}
				\rput(4.5,4.5){s}
				\rput(3.5,6.5){t}
				\rput(4.5,6.5){u}
			\end{pspicture}
			\end{center}
Again we have either $a=1$ and $b=2$ or we have $a=2$ and $b=1$. Due to symmetry we assume $a=1$ and $b=2$. Thus $c=2$, $d=1$, $e=2$, $f=1$, $g=2$, $h=1$ (dashed). Now we must have $i\geq 2$, but $i=2$ would lead to $j=k=1$, thus $i=3$ and $l=1$ (dotted). Similarly, we have $m=3$, $n=1$, $o=1$ (also dotted). Now we show that $p\leq 4$ can't be filled: $p=1$ is not allowed. If $p=2$ then $q=1$ and $r$ can't be filled. If $p=3$ then $s=2$ and $tu$ can't be filled. If $p=4$ then $s$ can't be filled.

In the remaining center part of (III) we start by inserting the results from the beginning of the center part of (II) (continuous lines in the bottom/right part):
	    
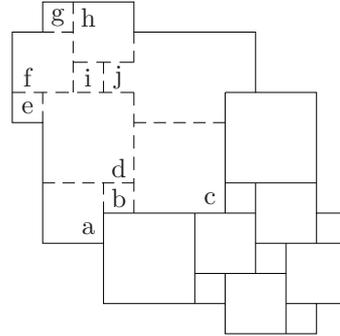
\begin{center}
			\begin{pspicture}(-2,0)(9,11)
				\pspolygon(1,1)(5,1)(5,0)(8,0)(8,1)(9,1)(9,4)(8,4)(8,8)(6,8)(6,10)(2,10)(2,11)(-1,11)(-1,10)(-2,10)(-2,7)(-1,7)(-1,3)(1,3)(1,1)
				\psline(4,1)(4,4)
				\psline(6,4)(4,4)
				\psline(6,2)(6,5)
				\psline(8,5)(6,5)
				\psline(6,3)(9,3)
				\psline(8,3)(8,4)
				\psline(7,0)(7,3)
				\psline(7,1)(8,1)
				\psline(4,2)(7,2)
				\psline(5,1)(5,2)				
				\rput(0.5,3.5){a}
				\rput(1.5,4.5){b}
				\rput(4.5,4.5){c}
				\rput(1.5,5.5){d}
				\rput(-1.5,7.5){e}
				\rput(-1.5,8.5){f}
				\rput(-0.5,10.5){g}
				\rput(0.5,10.5){h}
				\rput(0.5,8.5){i}
				\rput(1.5,8.5){j}
				\psline(1,4)(4,4)
				\psline(1,3)(1,4)
				\psline(5,4)(5,8)
				\psline(5,5)(6,5)
				\psline(5,8)(6,8)
				\psline[linestyle=dashed](-1,5)(2,5)
				\psline[linestyle=dashed](1,5)(1,4)
				\psline[linestyle=dashed](2,4)(2,8)
				\psline[linestyle=dashed](2,7)(5,7)
				\psline[linestyle=dashed](2,8)(-2,8)
				\psline[linestyle=dashed](-1,8)(-1,7)
				\psline[linestyle=dashed](0,8)(0,11)
				\psline[linestyle=dashed](-1,10)(0,10)
				\psline[linestyle=dashed](0,9)(2,9)
				\psline[linestyle=dashed](2,9)(2,10)
				\psline[linestyle=dashed](1,8)(1,9)
			\end{pspicture}
			\end{center}
First we conclude $a=2$, then $b=1$, then $c=3$, $d=3$, $e=1$, $f=2$, $g=1$, $h=2$, $i=1$ (dashed lines). But now $j$ can't be filled.
Thus we finally have proved that there is no squared square with size 17x17.
\end{proof}

\section{Final theorem}
\label{Sect6}
\begin{framed}\begin{thm}There are exactly two squared squares (and their mirrored versions) up to and including size 17x17: One of size 11x11 and one of size 16x16 -- both are shown in Figure 1 (page \pageref{fig1116}).\end{thm}\end{framed}
\begin{proof}
Combine Lemma \ref{least11}, \ref{lem11}, \ref{no12}, \ref{no13}, \ref{no14}, \ref{no15}, \ref{uni16} and \ref{no17}.
\end{proof}

\begin{acknowledgments}
I am grateful to Stuart~Anderson, Gregor~Milla and Kolja~Szillat for useful discussions.
\end{acknowledgments}

\end{document}